\documentclass{article}

\usepackage{arxiv}

\usepackage[utf8]{inputenc} 
\usepackage[T1]{fontenc}    
\usepackage{amsmath,amssymb,amsthm,color}
\usepackage{bm,bbm}
\usepackage{mathtools}
\usepackage{hyperref}       
\usepackage{url}            
\usepackage{amsfonts}       
\usepackage{microtype}      
\usepackage{graphicx}
\usepackage{doi}

\newtheorem{theorem}{Theorem}
\newtheorem*{theorem*}{Theorem}
\newtheorem{lemma}[theorem]{Lemma}
\newtheorem{proposition}[theorem]{Proposition}
\newtheorem{corollary}[theorem]{Corollary}

\theoremstyle{remark}

\newtheorem*{claim*}{\bf Claim}

\author{\href{https://orcid.org/0000-0003-1932-5115}{Mahmoud  Abdelgalil} \\
	Mechanical and Aerospace Engineering\\
	University at Buffalo, SUNY\\
	Buffalo, NY 14260 \\
	\texttt{maabdelg@buffalo.edu} 
	\And
	\href{https://orcid.org/0000-0003-0012-5447}{Tryphon T.~Georgiou} \\
	Mechanical and Aerospace Engineering\\
	University of California, Irvine\\
	Irvine, CA 92697\\
	\texttt{tryphon@uci.edu}
}

\date{}

\renewcommand{\headeright}{}
\renewcommand{\undertitle}{}
\renewcommand{\shorttitle}{}

\newcommand{\M}{M}
\newcommand{\Cinf}{C^{\infty}(\M)}
\newcommand{\CinfV}[1]{C_{#1}^{\infty}(\M)}
\newcommand{\cconv}{$\tfrac{1}{2}d^2$-convex}
\newcommand{\X}{\mathfrak{X}(\M)}
\newcommand{\Diff}{\mathrm{Diff}_0(\M)}
\newcommand{\SDiff}{\mathrm{SDiff}(\M)}
\newcommand{\Dens}{\mathrm{Dens}(\M)}
\DeclareMathOperator{\grad}{\mathrm{grad}}
\DeclareMathOperator{\expg}{\mathrm{exp}^{\text{$g$}}}
\DeclareMathOperator{\Div}{div}

\DeclareMathOperator{\Hess}{Hess}
\newcommand{\hor}{\mathrm{Hor}}
\newcommand{\GL}{\mathrm{GL}^+(\mathbb{R}^n)}
\newcommand{\Sym}{\mathrm{Sym}(n)}
\newcommand{\Skew}{\mathrm{Skew}(n)}
\title{A Ballantine-type factorization of diffeomorphisms}

\begin{document}
\maketitle

\begin{abstract}
We show that on a compact, connected, Riemannian manifold without boundary, every isotopic to the identity diffeomorphism can be written as a finite composition of optimal mass transport maps. This result represents a non-linear infinite dimensional generalization of Ballantine's factorization results.
\end{abstract}

\keywords{Optimal Mass Transport, Group of Diffeomorphisms.}

\section{Introduction}

In a series of publications culminating in \cite {ballantine1968bproducts,Ballantine1970}, Ballantine provides precise characterizations of the classes of matrices, real and complex, that can be written as the product of a given number of positive factors. Notably, he showed that, regardless of dimension, every matrix of positive determinant can be written as the product of \emph{five} positive definite factors, symmetric in the real case \cite{ballantine1968bproducts} and Hermitian in the complex case \cite{Ballantine1970}. The uniform upper bound on the required number of factors turns out to be instrumental in solving a fundamental problem in linear feedback theory \cite{AbdelgalilGeorgiou2026TAC}. 
In a recent manuscript \cite{AbdelgalilGeorgiou2026}, we offer a fresh perspective on Ballantine's results through the lens of optimal mass transport (OMT) \cite{Villani2009}. Specifically, by interpreting symmetric positive definite matrices as OMT maps between non-degenerate centered Gaussian distributions, we recognized that Ballantine's result in the real case is a statement on the extent to which OMT maps fail to commute in the linear category. Equivalently, his result implies that the \emph{holonomy} group in the restriction of Otto's principal bundle \cite{otto2001geometry} to the linear category \cite{modin2017geometry} is the entire structure group \cite{AbdelgalilGeorgiou2025}. 
The close analogy between the linear and smooth categories \cite{modin2017geometry,khesin2024information} naturally invites an investigation into whether a Ballantine-type factorization statement holds for the latter, i.e., whether every diffeomorphism in the identity component of the diffeomorphism group of a manifold $\M$, henceforth denoted by $\Diff$, can be written as the composition of a finite number of OMT maps. Our earlier work \cite[Theorem 7]{abdelgalil2026holonomy} proves a ``soft'' version of this statement, i.e., that, on a compact connected Riemannian manifold without boundary, the group generated by diffeomorphic OMT maps is a dense subset of $\Diff$ in the $C^\infty$-topology. The main purpose of the present manuscript is to upgrade this density statement into an exact factorization result. Namely, we prove the following.
\begin{theorem}\label{thm:OMT_DIFF}
    Let $(\M,g)$ be a compact, connected Riemannian manifold without boundary. Then, there exists an integer $r\in\mathbb{N}$ that depends solely on the minimal immersion dimension of $\M$, and an open neighborhood $\mathcal{O}\subset\Diff$ of the identity, such that every diffeomorphism in $\mathcal{O}$ is a composition of, at most, $r$ OMT maps.
\end{theorem}
Interestingly, Theorem \ref{thm:OMT_DIFF} imposes no restrictions on the topology or the curvature of $(\M,g)$. The theorem also shows that the number of needed factors is uniformly bounded in the open neighborhood $\mathcal{O}$. Whether this uniformity extends to all of $\Diff$ is an open question.

The remainder of this manuscript is organized as follows. In Section \ref{sec:prelim}, we fix our notation and collect some standard facts that we use in the sequel. In Section \ref{sec:corollaries}, we discuss some immediate corollaries of Theorem \ref{thm:OMT_DIFF}. Section \ref{sec:prelude} contains a finite-dimensional prelude, i.e., the linear case, which the proof of Theorem \ref{thm:OMT_DIFF} in Section \ref{sec:technical} mirrors. 

\section{Preliminaries}\label{sec:prelim}
Throughout, $(\M, g)$ is a smooth, compact, connected $n$-dimensional Riemannian manifold without boundary, and $\mu_g$ is its Riemannian measure, normalized to unit total mass. We denote by $N$ the \emph{minimal immersion dimension} of $\M$, i.e., the least $N\in\mathbb{N}$ for which a smooth immersion $\M\rightarrow\mathbb{R}^N$ exists. By Whitney's immersion theorem \cite[Theorem 6.20]{Lee2012}, $N\leq 2n-1$ whenever $n\geq2$.
We write $\Cinf$ and $\X$ for smooth functions and vector fields on $\M$ and, for any $\ell\in\mathbb{N}$, $\CinfV{\ell}$ for the smooth sections of the trivial vector bundle $\M\times\mathbb{R}^\ell$, which we identify with $\ell$-tuples of smooth functions
\[
\bm{\phi}:=(\phi^1,\dots,\phi^\ell).
\]
Tuples are written in boldface throughout, and the concatenation of $\bm{\phi}\in\CinfV{\ell}$ and $\bm{\eta}\in\CinfV{k}$ is written $(\bm{\phi},\bm{\eta})\in\CinfV{\ell+k}$. More generally, $C^\infty(\M,E)$ denotes smooth sections of a vector bundle $E\rightarrow\M$. We use $\grad$, $\nabla$, and $\Div$ to denote the gradient, the Levi-Civita connection, and the divergence of $(\M,g)$. Pointwise norms of tensors with respect to $g$ are written $|\cdot|_g$, and, for $k\in\mathbb{N}$, the $C^k$-norm of a tensor field $A$ is
\begin{align*}\textstyle
    \|A\|_{C^k}:=\sum_{j=0}^{k}\,\sup_{\M}|\nabla^jA|_g.
\end{align*}
For any integer $s\geq0$, $H^s(\M)$, $H^s(T\M)$, and $H_m^s(\M)$ denote the completions of $\Cinf$, $\X$, and $\CinfV{m}$ with respect to the Sobolev norm
$$\textstyle
\|A\|_{H^s}^2:=\sum_{j=0}^{s}\textstyle\int_{\M}|\nabla^jA|_g^2\,d\mu_g.
$$
For a family of tensor fields $A_\varepsilon$ indexed by $\varepsilon>0$, we write $A_\varepsilon=O(\varepsilon^k)$ if, for every $j\in\mathbb{N}$, there exists $C_j>0$ such that $\|A_\varepsilon\|_{C^j}\leq C_j\,\varepsilon^k$ for all sufficiently small $\varepsilon$.
We identify a Riemannian metric $h$ on $\M$ with the bundle isomorphism $h:T\M\rightarrow T^*\M$, $X\mapsto h(X,\cdot)$, whose inverse is $h^{-1}:T^*\M\rightarrow T\M$. For the reference metric $g$, the resulting \emph{musical isomorphisms} are denoted by $\flat$ and $\sharp$, so that $\grad\psi=(d\psi)^\sharp$. For a $(0,2)$-tensor $B$, $g^{-1}B$ is the endomorphism characterized by the relation $g(g^{-1}B\,X,Y)=B(X,Y)$. An endomorphism field $P$ is \emph{self-adjoint} with respect to $g$ if, for any $X,Y\in\X$, $g(PX,Y) = g(PY,X)$, and is positive definite if it is self-adjoint and $g_x(P_xv,v)>0$ for all $x\in\M$ and all $v\in T_x\M\backslash\{0\}$. The identity endomorphism of $T\M$, like the identity matrix, is denoted by $I$, whereas $\mathrm{id}\in\Diff$ denotes the identity diffeomorphism and $\mathrm{Id}_V$ the identity operator on $V$.
The Hessian of a function $\psi\in\Cinf$ is
\begin{align*}
    \Hess\psi(X,Y) = g(\nabla_X\grad\psi,Y),
\end{align*}
defined for any $X,Y\in \X$, and since the connection is torsion-free, it holds that $\Hess\psi(X,Y) = \Hess\psi(Y,X)$. The Lie bracket of vector fields is 
$[X,Y]=\nabla_XY-\nabla_YX$, and $\mathfrak{L}_Z$ denotes the Lie derivative along $Z\in\X$. Since $(\mathfrak{L}_Zg)(X,Y)=g(\nabla_XZ,Y)+g(X,\nabla_YZ)$, we have, for any $\psi\in\Cinf$,
\begin{align}\label{eq:killing}
    \mathfrak{L}_{\grad\psi}\,g=2\Hess\psi.
\end{align}
For any Riemannian metric $h$ on $\M$, we define the endomorphism field
\begin{align*}
    P_h&:=h^{-1}\circ g, & h(X,Y)&=g(P_h^{-1}X,Y),
\end{align*}
which is self-adjoint and positive with respect to $g$, and the $h$-gradient 
\[\grad_h\psi:=h^{-1}(d\psi)=P_h\grad\psi.\]
In particular, $\grad_g=\grad$ and $P_g=I$. On $\X$ and on $\CinfV{m}$, we use the $L^2(\mu_g)$ pairings
\begin{align*}
    \langle X,Y\rangle_{L^2}&:=\textstyle\int_{\M}g(X,Y)\,d\mu_g, & \langle\bm{\delta},\bm{\eta}\rangle_{L^2}&:=\textstyle\sum_{i=1}^{m}\int_{\M}\delta^i\eta^i\,d\mu_g,
\end{align*}
and $A^*$ denotes the formal adjoint of an operator $A$ with respect to these pairings.
For $\Phi\in\Diff$, $D\Phi:T\M\rightarrow T\M$ denotes the tangent map, $\Phi^*h:=h(D\Phi\,\cdot,D\Phi\,\cdot)$ and $\Phi_*h:=(\Phi^{-1})^*h$ the pullback and the pushforward of a metric $h$, $\Phi_*X:=D\Phi\circ X\circ\Phi^{-1}$ the pushforward of a vector field, and $\Phi_\sharp\mu$ the pushforward of a measure. For any metric $h$ and any $\psi\in\Cinf$,
\begin{align}\label{eq:naturality}
    \Phi_*(\grad_h\psi)=\grad_{\Phi_*h}(\psi\circ\Phi^{-1}).
\end{align}
We denote the Riemannian exponential map of $(\M,g)$ by $\expg$, and, for any $X\in\X$, write
$$
\expg(X):x\mapsto\exp_x^g(X(x)).
$$
As is well-known \cite{McCann2001}, an OMT map is of the form $\expg(\grad\phi)$ for a $\tfrac{1}{2}d^2$-convex\footnote{See the definition of $\tfrac{1}{2}d^{2}$-concave functions in \cite[Section 3.3]{rachev1998mass}, \cite{McCann2001}, from which $\phi$ is \cconv{} if $-\phi$ is $\tfrac{1}{2}d^{2}$-concave.} potential $\phi$, where $d$ is the Riemannian distance function. By \cite[Theorem 13.5]{Villani2009}, there exists $R>0$ such that every $\phi\in\Cinf$ with $\|\phi\|_{C^2}\leq R$ is \cconv{} and, for such $\phi$, McCann's \emph{displacement interpolation} \cite{McCann1997}
\begin{align*}
    \Phi^\phi_t:=\expg(t\grad\phi), \quad t\in[-1,1],
\end{align*}
is a smooth curve in $\Diff$ with $\Phi^\phi_0=\mathrm{id}$ \cite{McCann2001}. The associated Eulerian velocity field $\dot{\Phi}^\phi_t\circ(\Phi^\phi_t)^{-1}$ is the gradient of the potential $\alpha^\phi_t$ that solves the Hamilton--Jacobi equation
\begin{align}\label{eq:HJ}
    \partial_t\alpha^\phi_t+\tfrac{1}{2}|\grad\alpha^\phi_t|_g^2&=0, & \alpha^\phi_0&=\phi.
\end{align}
The \emph{conjugate potential} of $\phi$ is
\begin{align}\label{eq:ctransform}
    \phi^c:=-\alpha^\phi_1=-\bigl(\phi+\tfrac{1}{2}|\grad\phi|_g^2\bigr)\circ\expg(\grad\phi)^{-1},
\end{align}
which is the $\tfrac{1}{2}d^2$-transform of $\phi$ \cite[Chapter 5]{Villani2009}. The inverse of an OMT map is again an OMT map, i.e.,
\begin{align}\label{eq:inverse_OMT}
    \expg(\grad\phi)^{-1}=\expg(\grad\phi^c),
\end{align}
and, by \eqref{eq:ctransform} and the tame calculus, $\phi\mapsto\phi^c$ is continuous in the $C^\infty$-topology, with $0^c=0$.
For $\ell\in\mathbb{N}$, we define
\begin{align}\label{eq:Oell}
    O_\ell:=\{\bm{\phi}\in\CinfV{\ell}~|~\|\phi^i\|_{C^2}<R,\ \forall i\in\{1,\dots,\ell\}\},
\end{align}
which is an open set, and, for $\bm{\phi}\in O_\ell$, the tuple of conjugate potentials
\begin{align*}
    \bm{\phi}^c:=\bigl((\phi^\ell)^c,\dots,(\phi^1)^c\bigr)\in\CinfV{\ell},
\end{align*}
the reversal dictated by the fact that the inverse of a composition is the composition of the inverses in reverse order.
The family of semi-norms $\{\|\cdot\|_{C^k}\}_{k\in\mathbb{N}}$ defines the $C^\infty$-topology. The identity component $\Diff$ of the group of smooth diffeomorphisms, equipped with the $C^\infty$-topology, is a Fr\'echet Lie group \cite[Part~I, Section 4]{Hamilton1982}, whose tangent bundle is trivialized by right translation, i.e.,
\begin{align*}
    T_{\Phi}\Diff&=\{X\circ\Phi~|~X\in\X\}, & T_{\mathrm{id}}\Diff&=\X.
\end{align*}
We write $\dot{\Phi}$ for a generic element of $T_\Phi\Diff$. For a map $F$ between Fr\'echet manifolds, $dF(u)[v]$ denotes its derivative at $u$ in the direction $v$.
\section{Corollaries to Theorem \ref{thm:OMT_DIFF}}\label{sec:corollaries}
Before we proceed to the proof of Theorem \ref{thm:OMT_DIFF}, we discuss some of its immediate implications. 
Since any subgroup of a connected topological group that contains an open neighborhood of the identity is the entire group, Theorem \ref{thm:OMT_DIFF} immediately implies the following.
\begin{corollary}
    Every diffeomorphism in $\Diff$ is a finite composition of diffeomorphic OMT maps.
\end{corollary}
As discussed in Section \ref{sec:prelim}, each OMT map with a sufficiently small smooth potential is joined to $\mathrm{id}$ by McCann's displacement interpolation. The associated Eulerian velocity field is the gradient of a time-varying potential that solves the Hamilton--Jacobi equation \eqref{eq:HJ}. Concatenating these curves gives the following consequence of Theorem \ref{thm:OMT_DIFF}.
\begin{corollary}\label{cor:controllability}
For any $\Phi_{\rm fin}\in\Diff$ and any $T>0$, there exists a piecewise smooth curve of potentials $t\mapsto\alpha_t$, $t\in[0,T]$, such that the solution to the initial value problem
    \begin{align}\label{eq:controlsystem}
        \dot{\Phi}_t &= \grad\alpha_{t} \circ \Phi_t, & \Phi_0=\mathrm{id},
    \end{align}
    satisfies the endpoint constraint $\Phi_T= \Phi_{\rm fin}$. Moreover, if $\Phi_{\rm fin}\in\mathcal{O}$, then the total number of switches in the piece-wise smooth curve $t\mapsto\alpha_t$ may always be taken to be less than or equal to $r$.
\end{corollary}
The above result is a statement on the \emph{controllability} of the infinite-dimensional right-invariant control system \eqref{eq:controlsystem}. As such, it is natural to contrast our results with the seminal work of Agrachev and Caponigro \cite{AgrachevCaponigro2009}. A quick glance at the structure of admissible vector fields in \eqref{eq:controlsystem} reveals the key difference: the family of admissible vector fields in \cite{AgrachevCaponigro2009} is closed under multiplication by smooth functions, whereas gradient vector fields are not. Hence, the results in \cite{AgrachevCaponigro2009}, see also \cite{ArguillereTrelat2017}, are not applicable to our setting. 
Another corollary of Theorem \ref{thm:OMT_DIFF} relates to the \emph{holonomy} group of the underlying principal bundle in Otto's geometric framework \cite{otto2001geometry}. It is well-known (see, e.g., \cite[Appendix A.5]{WendtKhesin2009}, and \cite[Part~III, Theorem 2.3.5]{Hamilton1982}) that the map
\begin{align*}
    \pi&:\Diff\rightarrow\Dens, & \pi(\Phi):=\frac{d(\Phi_\sharp \mu_g)}{d\mu_g},
\end{align*}
turns $\Diff$ into a principal bundle over the space of smooth and strictly positive probability densities
\begin{align*}
    \Dens:=\{\rho\in \Cinf ~|~\rho>0,\,\textstyle\int\rho\,d\mu_g = 1\}.
\end{align*}
The structure group of this bundle, acting by right composition, is
\begin{align*}
    \SDiff:=\{\Phi\in\Diff~|~\Phi_\sharp \mu_g = \mu_g\},
\end{align*}
which coincides with the identity component of the group of all measure-preserving diffeomorphisms of $\M$. Otto's construction equips this bundle with a principal connection whose horizontal space at $\Phi\in\Diff$ is
$$
\hor_{\Phi}:=\{\grad\phi\circ\Phi~|~\phi\in\Cinf\}.
$$
In this framework, the parallel transport map along a Wasserstein geodesic is left composition with the associated OMT map. Consequently, parallel transport along a piecewise-geodesic path in $\Dens$ is left composition with the corresponding composition of OMT maps. Theorem \ref{thm:OMT_DIFF} then immediately yields the following.
\begin{corollary}\label{cor:holonomy}
    Every element of $\SDiff$ is the holonomy of a piecewise-geodesic loop in $\Dens$.
\end{corollary}

\section{A Finite-Dimensional Prelude}\label{sec:prelude}
In this section, we give a preview of the strategy for the proof of Theorem \ref{thm:OMT_DIFF} by considering the finite-dimensional linear setting. Throughout this section, $\GL$ denotes the identity component of the general linear group, $I$ is the identity matrix, $\Sym,\Skew\subset\mathbb{R}^{n\times n}$ are the vector space of symmetric, respectively, skew-symmetric, matrices, and $e^A$ is the standard matrix exponential. 
\begin{proposition}\label{prop:OMT_GL}
    There exists an open neighborhood $\mathcal{O}\subset\GL$ of $I$ such that every $A\in \mathcal{O}$ is the product of, at most, three symmetric positive definite factors.
\end{proposition}
\begin{proof}
    We introduce the map
    \begin{align*}
        \mathrm{End}_2&:\Sym\times\Sym\rightarrow \GL, & \mathrm{End}_2(S_1,S_2):= e^{S_1}e^{S_2}.
    \end{align*}
    An elementary computation shows that the differential of $\mathrm{End}_2$ is given by
    \begin{align*}
        d\mathrm{End}_2(S_1,S_2)[Q_1,Q_2] &= \left(\mathcal{A}_{S_1}(Q_1) + e^{S_1}\mathcal{A}_{S_2}(Q_2)e^{-S_1}\right)\mathrm{End}_2(S_1,S_2),
    \end{align*}
    where the operator $\mathcal{A}_S$ is defined by
    \begin{align*}
        \mathcal{A}_S(Q)&:=\int_0^1 e^{t S} Q e^{-t S}\,dt.
    \end{align*}
    Taking $S_1 = 0$, $S_2=\varepsilon S$, for any $\varepsilon>0$ and $S\in\Sym$, the differential simplifies to
    \begin{align*}
        d\mathrm{End}_2(0,\varepsilon S)[Q_1,Q_2] = \left(Q_1+\mathcal{A}_{\varepsilon S}(Q_2)\right)\mathrm{End}_2(0,\varepsilon S).
    \end{align*}
    The linear operator $\mathcal{A}_{\varepsilon S}$ admits the expansion
    \begin{align*}
        \mathcal{A}_{\varepsilon S}(Q) = Q + \frac{1}{2}\varepsilon\,\mathrm{ad}_S(Q) + \varepsilon^2\mathcal{B}_{\varepsilon}(Q),
    \end{align*}
    where $\mathrm{ad}_S:\mathfrak{gl}(n)\rightarrow\mathfrak{gl}(n)$ is the standard \emph{adjoint} operator $\mathrm{ad}_S(Q)=[S,Q]=SQ-QS$ on the Lie algebra $\mathfrak{gl}(n)$ of $\GL$, and $\mathcal{B}_{\varepsilon}$ is a bounded linear operator. This can be shown using \cite[Proposition 3.35]{hall2003lie} and direct integration. Reparametrizing  $Q_2=2\varepsilon^{-1}\tilde{Q}_2$ and $Q_1=\tilde{Q}_1-Q_2$ with $\tilde{Q}_1,\tilde{Q}_2\in\Sym$ and $\varepsilon>0$, we obtain that
    \begin{align*}
        d\mathrm{End}_2(0,\varepsilon S)[Q_1,Q_2] = (\tilde{Q}_1 + \mathrm{ad}_S(\tilde{Q}_2) + 2\varepsilon\mathcal{B}_{\varepsilon}(\tilde{Q}_2))\mathrm{End}_2(0,\varepsilon S),
    \end{align*}
    where we utilized the linearity of $\mathcal{B}_{\varepsilon}$. We claim that there exist $S\in\Sym$ and $\varepsilon_S>0$ such that, for all $\varepsilon\in (0,\varepsilon_S)$, the linear map defined by
    \begin{align*}
        \Sym\times\Sym\ni(\tilde{Q}_1,\tilde{Q}_2)\mapsto \tilde{Q}_1 + \operatorname{ad}_S(\tilde{Q}_2) + 2\varepsilon\mathcal{B}_{\varepsilon}(\tilde{Q}_2)\in\mathfrak{gl}(n)
    \end{align*}
    is surjective. Because $\mathcal{B}_{\varepsilon}$ is uniformly bounded for a fixed $S\in\Sym$, and because surjectivity is stable under small perturbation, the above claim follows from the surjectivity of the linear map
   $(\tilde{Q}_1,\tilde{Q}_2)\mapsto \tilde{Q}_1 + \operatorname{ad}_S(\tilde{Q}_2)$
    for some $S\in\Sym$. Moreover, since $\operatorname{ad}_S(Q)$ is skew-symmetric for any $S,Q\in\Sym$, the statement further reduces to the surjectivity of $\mathrm{ad}_S$ onto $\Skew\subset\mathfrak{gl}(n)$, for some $S\in\Sym$. Because $S$ is symmetric, the adjoint of $\mathrm{ad}_S$, denoted by $\mathrm{ad}_S^*$, is again given by $\mathrm{ad}_S^*(\Omega)=[S,\Omega]$. Hence, $\mathrm{ad}_S\mathrm{ad}_S^*(\Omega) = [S,[S,\Omega]]$. Taking $S = \mathrm{diag}(\lambda_1,\dots,\lambda_n)$, it is not difficult to see that entries of $\mathrm{ad}_S\mathrm{ad}_S^*(\Omega)$ are
    \begin{align*}
        (\mathrm{ad}_S\mathrm{ad}_S^*(\Omega))_{ij} = (\lambda_i-\lambda_j)^2 \Omega_{ij},
    \end{align*}
    so that $\mathrm{ad}_S\mathrm{ad}_S^*(\Omega) = \Lambda\odot \Omega$, where $(\Lambda)_{ij}=(\lambda_i-\lambda_j)^2$ and $\odot$ is the Hadamard product of matrices. In particular, $\mathrm{ad}_S\mathrm{ad}_S^*$ is invertible on $\Skew$ if and only if the eigenvalues $\{\lambda_1,\dots,\lambda_n\}$ are distinct. Picking any $S\in\Sym$ that satisfies this generic condition and fixing any $\varepsilon\in(0,\varepsilon_S)$, the spanning condition follows, and, since $\mathrm{End}_2(0,\varepsilon S)$ is invertible, we obtain that the differential is surjective. Hence, the submersion theorem implies that the image of $\mathrm{End}_2$ contains an open neighborhood of $\mathrm{End}_2(0,\varepsilon S)=e^{\varepsilon S}$ in $\GL$. Through right multiplication by $e^{-\varepsilon S}$, which is a symmetric positive definite matrix, we obtain an open neighborhood $\mathcal{O}$ of $I$ such that every element in $\mathcal{O}$ is of the form $e^{S_1}e^{S_2}e^{-\varepsilon S}$, i.e., a product of three symmetric positive definite factors.
\end{proof}
The proof of Proposition \ref{prop:OMT_GL} contains the essence of the proof of Theorem \ref{thm:OMT_DIFF}. In particular, the trick of relocating to a point whose image is a suitable near-identity point, showing the surjectivity of the differential, and translating back via right multiplication, will carry over mutatis mutandis. Note that, when $n\geq2$, the differential is not surjective at the origin $(0,\dots,0)$, regardless of the number of factors. Indeed, the image under the differential at this point is $\Sym$ only. This degeneracy also carries over to the setting of Theorem~\ref{thm:OMT_DIFF}.

Unlike the finite-dimensional case, however, the group $\Diff$ with the $C^\infty$-topology is a Fr\'echet manifold rather than a Banach one. As such, the classical inverse function theorem is not applicable, and we shall rely instead on the Nash--Moser theorem in the tame category \cite{Hamilton1982}.
\section{Proof of Theorem \ref{thm:OMT_DIFF}}\label{sec:technical}
Consider the map obtained by concatenating the Riemannian exponentials of the gradients of the potentials $\phi^i$, i.e., 
\begin{align*}
    \mathrm{End}_{\ell}:(\phi^1,\dots,\phi^\ell)\mapsto \exp^g(\grad\phi^1)\circ\cdots \circ\exp^g(\grad\phi^\ell).
\end{align*}
For $\bm{\phi}\in\CinfV{\ell}$ and $i\in\{0,\dots,\ell\}$, we also write $\mathrm{End}_{i}(\bm{\phi}):=\exp^g(\grad\phi^1)\circ\cdots\circ\exp^g(\grad\phi^i)$ for the partial compositions, with the convention $\mathrm{End}_{0}(\bm{\phi}):=\mathrm{id}$. Restricted to the open set $O_\ell$ of \eqref{eq:Oell}, the map $\mathrm{End}_{\ell}$ is a composition of {\em optimal transport maps}, and $\mathrm{End}_{\ell}(O_\ell)\subset\Diff$.
Our goal is to show that, for some $\ell\in\mathbb{N}$, $\mathrm{End}_{\ell}(O_\ell)$ contains an open neighborhood of $\mathrm{id}$ in $\Diff$. 
Our strategy mirrors the finite-dimensional approach. First, we will show that there exists an $m\in\mathbb{N}$ and a tuple of potentials $\bm{\phi}_{\star}\in O_m$, such that the differential $d\mathrm{End}_{m}$ admits a tame right inverse in an open neighborhood of $\bm{\phi}_{\star}$ and, in addition, that the tuple of conjugate potentials satisfies $\bm{\phi}_{\star}^c\in O_m$.
The Nash--Moser theorem then implies that the image under $\mathrm{End}_{m}$ of every open neighborhood of $\bm{\phi}_{\star}$ in $O_m$ contains an open neighborhood of $\mathrm{End}_{m}(\bm{\phi}_{\star})$ in $\Diff$. 
Moreover,
\[
\mathrm{End}_{m}(\bm{\phi}_\star)\circ\mathrm{End}_{m}(\bm{\phi}_\star^c)=\mathrm{id},
\]
by \eqref{eq:inverse_OMT}, while at the same time, $(\bm{\phi}_\star,\bm{\phi}_\star^c)\in O_{2m}$, so that $\mathrm{End}_{2m}(O_{2m})$ contains an open neighborhood of the identity, and we achieve our goal with $\ell = 2m$. 
To that end, we have the following lemma, the proof of which is a long but straightforward computation. 

\begin{lemma}
    For any $m\in\mathbb{N}$, any $\bm{\phi}\in O_m$, and any $\bm{\delta}\in\CinfV{m}$,
    \begin{align*}
        d\mathrm{End}_{m}(\bm{\phi})[\bm{\delta}]&=\left(\sum_{i=1}^{m} \grad_{h_i(\bm{\phi})}\big( \delta^i\circ \mathrm{End}_{i}(\bm{\phi})^{-1}\big)\right)\circ \mathrm{End}_{m}(\bm{\phi}),
    \end{align*}
    where $h_i(\bm{\phi})$ are the Riemannian metrics defined by
    \begin{align*}
        h_i(\bm{\phi})&:=(\mathrm{End}_{i}(\bm{\phi}))_*\tilde{g}_{\phi^i}, & \tilde{g}_{\phi^i} &:= \left(\int_0^1 \big(\expg(\tau \grad \phi^i)^*g\big)^{-1}\,d\tau\right)^{-1}.
    \end{align*}
\end{lemma}
\begin{proof}
    First, let us compute the differential of the one-factor map
\begin{align*}
    \mathrm{End}:\phi\mapsto \expg(\grad \phi), %
\end{align*}
which is the solution at time $t=1$ of the initial value problem
\begin{align*}
    \dot{\Phi}_t&= \grad \alpha_t \circ \Phi_t, & \Phi_0&=\mathrm{id},
\end{align*}
with $\alpha_t$ solving the Hamilton--Jacobi (HJ) equation \eqref{eq:HJ}.
Perturbing $\phi\in O_1$ to $\phi^{s}=\phi + s\,\delta \in O_1$ for any $\delta\in\Cinf$ and sufficiently small $|s|$, and linearizing the HJ equation, we obtain
\begin{align*}
    \partial_t(\partial_{s}\alpha^{s}_t) + g(\grad\alpha^{s}_t,\grad(\partial_{s}\alpha^{s}_t)) &= 0, & \partial_{s}\alpha^{s}_0=\delta, %
\end{align*}
so that, if $\delta_t = \partial_{s}\alpha^{s}_t|_{{s} = 0}$, then $\delta_t$ solves the transport equation %
\begin{align}\label{eq:HJ_variation}
    \partial_t\delta_t + g(\grad\alpha_t,\grad \delta_t) &= 0, & \delta_0=\delta.
\end{align}
Moreover, since \eqref{eq:HJ_variation} is a transport equation, its solution is given in closed form by $\delta_t = \delta\circ \expg(t \grad \phi)^{-1}$. Applying the non-homogeneous variations formula \cite[Section 2.8]{agrachev2013control}, we obtain that
\begin{align*}
    \partial_{s}|_{{s}=0}(\Phi_1^{{s}}) = \mathcal{A}_{\phi}(\delta)\circ \exp^g(\grad\phi),
\end{align*}
where the operator $\mathcal{A}$ is given by
\begin{align*}
    \mathcal{A}_{\phi}(\delta)
    &:= \int_0^1 \big(\exp^g(\grad \phi)\circ\exp^g(\tau \grad \phi)^{-1}\big)_*\grad\delta_\tau \,d\tau.
\end{align*}
With $\Phi_\tau = \exp^g(\tau \grad \phi)$, we compute that
\begin{align*} %
    \mathcal{A}_{\phi}(\delta)
    &= \int_0^1 \big(\Phi_1\circ\Phi_\tau^{-1}\big)_*\grad\delta_\tau \,d\tau= (\Phi_1)_*\int_0^1 \grad_{\big(\Phi_\tau^{-1}\big)_*g}(\delta_\tau\circ\Phi_\tau) \,d\tau.
\end{align*}
where we utilize the identity \eqref{eq:naturality}. Using $\delta_\tau = \delta\circ \Phi_\tau^{-1}$, we obtain
\begin{align*}
    \mathcal{A}_{\phi}(\delta) &= (\Phi_1)_*\int_0^1 \grad_{\big(\Phi_\tau^{-1}\big)_*g} \delta  \,d\tau =  (\Phi_1)_*\int_0^1 \big(\big(\Phi_\tau^{-1}\big)_*g\big)^{-1} d\delta  \,d\tau\\
    &=  (\Phi_1)_*\left(\left(\int_0^1 \big(\big(\Phi_\tau^{-1}\big)_*g\big)^{-1} \,d\tau\right) d\delta\right),
\end{align*}
and, for a fixed $\phi\in O_1$, $\mathcal{A}_{\phi}$ simplifies to
\begin{align*}
    \mathcal{A}_{\phi}(\delta) &= (\Phi_1)_*\grad_{\tilde{g}_\phi} \delta, & \tilde{g}_\phi &:= \left(\int_0^1 \big((\Phi_\tau)^*g\big)^{-1}\,d\tau\right)^{-1}.
\end{align*}
Next, we apply this to $\mathrm{End}_{m}$ using the chain rule. Specifically, with $\bm{\delta}=(\delta^{1},\dots,\delta^{m})\in\CinfV{m}$ and $\bm{\phi}=(\phi^1,\dots,\phi^m)\in O_m$, we consider
\begin{align*}
    d\mathrm{End}_{m}(\bm{\phi})[\bm{\delta}] = \partial_{s}|_{{s}=0}\left(\Phi_1^{1,{s}}\circ \cdots\circ \Phi_1^{m,{s}}\right), 
\end{align*}
where $\Phi_t^{i,{s}}$ solves 
$
\dot{\Phi}_t^{i,{s}} = \grad\alpha_t^{i,{s}} \circ \Phi_t^{i,{s}},$ for $\Phi_0^{i,{s}}=\mathrm{id}, 
$
with $\alpha_t^{i,{s}}$ solving \eqref{eq:HJ} from the initial datum $\alpha_0^{i,{s}}=\phi^i + {s}\,\delta^{i}$. 
The chain rule gives
\begin{align*}
    d&\mathrm{End}_{m}(\bm{\phi})[\bm{\delta}]\\
    &= \sum_{i=1}^{m}D\left(\Phi_1^{1,{s}}\circ \cdots\circ \Phi_1^{i-1,{s}}\right)\Big|_{{s}=0}\circ\left(\partial_{s}|_{{s}=0} \Phi_1^{i,{s}}\right) \circ \left(\Phi_1^{i+1,{s}}\circ \cdots\circ \Phi_1^{m,{s}}\right)\Big|_{{s}=0}\\ 
    & =\sum_{i=1}^{m}D(\mathrm{End}_{i-1}(\bm{\phi}))\circ\left(\partial_{s}|_{{s}=0} \Phi_1^{i,{s}} \right)\circ \mathrm{End}_{i}(\bm{\phi})^{-1}\circ \mathrm{End}_{m}(\bm{\phi}).
\end{align*}
Using the expression we derived for $\partial_{s}|_{{s} =0}\Phi^{s}_1$, we obtain that
\begin{align*}
    d&\mathrm{End}_{m}(\bm{\phi})[\bm{\delta}]\\
    &= \sum_{i=1}^{m}D(\mathrm{End}_{i-1}(\bm{\phi}))\circ\mathcal{A}_{\phi^i}(\delta^{i})\circ\exp^g(\grad\phi^i)\circ \mathrm{End}_{i}(\bm{\phi})^{-1}\circ \mathrm{End}_{m}(\bm{\phi}).
\end{align*}
Substituting for $\mathcal{A}_{\phi^i}$, we compute that
\begin{align*}
    d\mathrm{End}_{m}(\bm{\phi})[\bm{\delta}] &= \sum_{i=1}^{m}D(\mathrm{End}_{i}(\bm{\phi}))\circ (\grad_{\tilde{g}_{\phi^i}} \delta^i)\circ \mathrm{End}_{i}(\bm{\phi})^{-1}\circ \mathrm{End}_{m}(\bm{\phi}), \\
    &=\left(\sum_{i=1}^{m} \grad_{(\mathrm{End}_{i}(\bm{\phi}))_*\tilde{g}_{\phi^i}}\big( \delta^i\circ \mathrm{End}_{i}(\bm{\phi})^{-1}\big)\right)\circ \mathrm{End}_{m}(\bm{\phi})
\end{align*}
where the last equality utilizes \eqref{eq:naturality} once more.
\end{proof}
From the formula for $d\mathrm{End}_{m}(\bm{\phi})$, we see that the sought-after point $\bm{\phi}_\star$ cannot be taken to be $0$. Indeed,
\begin{align*}
    d\mathrm{End}_{m}(0)[\bm{\delta}] = \grad\Big(\sum_{i=1}^m \delta^i\Big),
\end{align*}
whose image consists of gradient vector fields only.
Taking any $\bm{\phi}\in O_m$, we observe that $d\mathrm{End}_{m}(\bm{\phi})$ can be factored as
\begin{align*}
    d\mathrm{End}_{m}(\bm{\phi}) = \mathcal{R}(\bm{\phi})\circ\mathcal{L}(\bm{\phi})\circ\mathcal{C}(\bm{\phi}), 
\end{align*}
where the families of operators $\mathcal{R}$, $\mathcal{L}$, $\mathcal{C}$ are defined by 
\begin{align*}
    \mathcal{R}&:O_m\times \X\rightarrow T\Diff, & \mathcal{R}(\bm{\phi})[X]&:= X\circ \mathrm{End}_{m}(\bm{\phi})\in T_{\mathrm{End}_{m}(\bm{\phi})}\Diff,\phantom{\sum_{i=1}^{m}}\\
    \mathcal{L}&:O_m\times\CinfV{m}\rightarrow\X, & \mathcal{L}(\bm{\phi})[\bm{\delta}]&:= \sum_{i=1}^{m} \grad_{h_i(\bm{\phi})} \delta^i, \\ \mathcal{C}&:O_m\times\CinfV{m}\rightarrow \CinfV{m}, & \mathcal{C}(\bm{\phi})[\bm{\delta}]&:= \bigl(\delta^1\circ\mathrm{End}_{1}(\bm{\phi})^{-1},\dots,\delta^m\circ\mathrm{End}_{m}(\bm{\phi})^{-1}\bigr).\phantom{\sum_{i=1}^{m}}
\end{align*}
In addition, for all $\bm{\phi}\in O_m$, the operators $\mathcal{R}(\bm{\phi})$ and $\mathcal{C}(\bm{\phi})$ are invertible, with explicit inverses
\begin{align*}
\mathcal{R}^{-1}(\bm{\phi})&:T_{\mathrm{End}_{m}(\bm{\phi})}\Diff \rightarrow \X, & \mathcal{R}^{-1}(\bm{\phi})[\dot{\Phi}]&= \dot{\Phi}\circ \mathrm{End}_{m}(\bm{\phi})^{-1}, \\
     \mathcal{C}^{-1}(\bm{\phi})&:\CinfV{m}\rightarrow\CinfV{m}, & \mathcal{C}^{-1}(\bm{\phi})[\bm{\delta}]&= \bigl(\delta^1\circ\mathrm{End}_{1}(\bm{\phi}), \dots, \delta^m\circ \mathrm{End}_{m}(\bm{\phi})\bigr).
\end{align*}
The tame calculus of Hamilton, i.e.,  that the composition map $(u,\Phi)\mapsto u\circ\Phi$ and the inversion map $\Phi\mapsto\Phi^{-1}$ are smooth tame maps \cite[Part~II, Theorem~2.3.5]{Hamilton1982}, implies that the corresponding families of operators
\begin{align*}
    \mathcal{R}^{-1}&: (\mathrm{End}_{m})^*T\Diff\rightarrow \X, \\ 
    \mathcal{C}^{-1}&:O_m\times\CinfV{m}\rightarrow\CinfV{m},
\end{align*}
are also smooth tame, where 
\begin{align*}
    (\mathrm{End}_{m})^*T\Diff:= \{(\bm{\phi},\dot{\Phi})\in O_m\times T\Diff~|~\dot{\Phi}\in T_{\mathrm{End}_{m}(\bm{\phi})}\Diff\}.
\end{align*}
Therefore, our task further reduces to showing that the family of operators $\mathcal{L}$ admits a tame right-inverse for all $\bm{\phi}$ in an open neighborhood of some suitable $\bm{\phi}_\star\in O_m$.
The existence of such a right inverse, denoted by $\mathcal{L}^\dagger$ below, is the subject of the following result, which also invokes the tame calculus of Hamilton \cite[Part~II, Section 3.3]{Hamilton1982}.
\begin{proposition}\label{prop:tame_inverse}
    Suppose that, for some $m\in\mathbb{N}$ and $\bm{\phi}_\star\in O_m$, the operator 
    $$
    \mathcal{L}(\bm{\phi}_\star)\circ \mathcal{L}^*(\bm{\phi}_\star):H^{s+2}(T\M)\rightarrow H^{s}(T\M),
    $$
    is elliptic and invertible for all $s\geq 0$, where $\mathcal{L}^*(\bm{\phi}_\star)$ denotes the adjoint of $\mathcal{L}(\bm{\phi}_\star)$ with respect to the $L^2(\mu_g)$ pairing. Then, there exists an open neighborhood $\mathcal{U}\subset O_m$ of $\bm{\phi}_\star$ such that, for all $\bm{\phi}\in \mathcal{U}$,
    $
    \mathcal{L}(\bm{\phi})\circ \mathcal{L}^*(\bm{\phi})
    $
    is also elliptic and invertible between the same spaces. Moreover, the family of operators 
    \begin{align*}
        \mathcal{L}^\dagger&:\mathcal{U}\times\X\rightarrow\CinfV{m}, & \mathcal{L}^\dagger(\bm{\phi})[X]&:= \mathcal{L}^*(\bm{\phi})\circ (\mathcal{L}(\bm{\phi})\circ \mathcal{L}^*(\bm{\phi}))^{-1}[X],
    \end{align*}
    is smooth tame, and, for all $\bm{\phi}\in \mathcal{U}$, 
    $$
    \mathcal{L}(\bm{\phi})\circ\mathcal{L}^\dagger(\bm{\phi}) = \mathrm{Id}_{\X}.
    $$
\end{proposition}
\begin{proof}
    For all $\bm{\phi}\in O_m$, with $P_i(\bm{\phi}):=P_{h_i(\bm{\phi})}=h_i(\bm{\phi})^{-1}\circ g$, the adjoint $\mathcal{L}^*(\bm{\phi})$ with respect to the $L^{2}(\mu_g)$ pairing is 
\begin{align*}
    \mathcal{L}^*(\bm{\phi})[X]= -\bigl(\Div(P_1(\bm{\phi})X),\dots,\Div(P_m(\bm{\phi})X)\bigr).
\end{align*}
Therefore, the operator $\mathcal{L}(\bm{\phi})\circ\mathcal{L}^*(\bm{\phi})$ is given by
\begin{align*}
    \mathcal{L}(\bm{\phi})\circ\mathcal{L}^*(\bm{\phi})[X] &= -\sum_{i=1}^m \grad_{h_i(\bm{\phi})}\bigl(\Div(P_i(\bm{\phi})X)\bigr)= - \sum_{i=1}^m P_i(\bm{\phi})\,\grad \circ\Div (P_i(\bm{\phi}) X),
\end{align*}
which is a second-order differential operator. Let $D^2(T\M,T\M)$ denote the bundle of linear partial differential operators of degree at most $2$ from $T\M$ to $T\M$ and let $\mathbb{L}$ be the generic differential operator map defined in \cite[Part~II, Section 3.3]{Hamilton1982}. Then, observe that the map $f$ defined by
\begin{align*}
    f&: O_m\rightarrow C^\infty(\M,D^2(T\M,T\M)), & \mathbb{L}(f(\bm{\phi}))&:=\mathcal{L}(\bm{\phi})\circ\mathcal{L}^*(\bm{\phi})
\end{align*}
takes a tuple $\bm{\phi}\in O_m$ into the \emph{coefficients} of a second-order differential operator. In addition, all the operations involved in defining the map $f$ are smooth tame maps, i.e., composition, inversion, Riemannian exponentiation, integration in $\tau$, and a finite number of differentiations. Therefore, $f$ itself is smooth tame. Now, let $U\subseteq C^\infty(\M,D^2(T\M,T\M))$ denote the open set of coefficients for which the corresponding operator is elliptic and invertible \cite[Part~II, Section 3.3]{Hamilton1982}. By assumption, we have $f(\bm{\phi}_\star)\in\mathcal{U}$, and, since $f$ is continuous, $\mathcal{U}:=f^{-1}(U)$ is a non-empty open neighborhood of $\bm{\phi}_\star$ contained in $O_m$ on which $\mathcal{L}(\bm{\phi})\circ\mathcal{L}^*(\bm{\phi})$ is elliptic and invertible. Moreover, by \cite[Part~II, Theorem 3.3.1]{Hamilton1982}, the solution map $\bm{\phi}\mapsto(\mathcal{L}(\bm{\phi})\circ\mathcal{L}^*(\bm{\phi}))^{-1}$ is smooth tame on $\mathcal{U}$, and, since $\mathcal{L}^*$ is a smooth tame family of first-order operators, $\mathcal{L}^\dagger$ is also smooth tame. The conclusion is immediate.
\end{proof}

Thanks to Proposition \ref{prop:tame_inverse}, our task further reduces to finding a single point $\bm{\phi}_\star$ that meets the conditions in its statement. The following result shows that, indeed, such conditions can be met.
\begin{proposition}\label{prop:ellipticity}
    Let $\beta^i:\M\rightarrow\mathbb{R}$, for $i\in\{1,\dots,N\}$, be the component functions of the minimal smooth immersion of $\M$. Set $m=1+2N + \tbinom{N+1}{2}$ and define $\psi^1=0$, and $\psi^{i}$, for $i\in\{2,\ldots,m\}$, as follows:
    \begin{align*}
        \psi^{i+1}&:=\beta^i, \mbox{ for }1\leq i\leq N\\
        \psi^{i+N+1}&:= (\beta^{i})^3, \mbox{ for }1\leq i\leq N\\
\psi^{2N+j+1}&:=\beta^{\iota(j)}, \mbox{ for }1\leq j\leq \tbinom{N+1}{2}
    \end{align*}
 with $\iota$ being any bijection between $\{1,\dots,\tbinom{N+1}{2}\}$ and $\{(k,l)~|~1\leq k\leq l\leq N\}$, and $\beta^{(k,l)} := \beta^k\beta^l$. Then, there exists some $\varepsilon > 0$ such that, with $\bm{\phi}_\star = \varepsilon \bm{\psi}$, we have that $\bm{\phi}_\star,\bm{\phi}_\star^c\in O_m$, and the linear operator
 $$
 \mathcal{L}(\bm{\phi}_\star)\circ\mathcal{L}^*(\bm{\phi}_\star):H^{s+2}(T\M)\rightarrow H^s(T\M),
 $$ 
 is elliptic and invertible for any $s\geq 0$. 
\end{proposition}
\begin{proof}
Throughout, let $\bm{\phi} = \varepsilon\, \bm{\psi} = (\varepsilon \,\psi^1,\dots,\varepsilon \,\psi^m)$ and note that there exists $\varepsilon_0>0$ such that, for all $\varepsilon\in(0,\varepsilon_0)$, we have that $\bm{\phi},\bm{\phi}^c\in O_m$.
From the expression of $\mathcal{L}(\bm{\phi})\circ\mathcal{L}^*(\bm{\phi})$ in the proof of Proposition \ref{prop:tame_inverse}, its principal symbol is
\begin{align*}
    \sigma(x,\xi) = \sum_{i=1}^m P_i(\bm{\phi})_x(\xi^\sharp \otimes \xi)P_i(\bm{\phi})_x = \sum_{i=1}^m (P_i(\bm{\phi})_x\xi^\sharp) \otimes (P_i(\bm{\phi})_x\xi^\sharp)^\flat,
\end{align*}
where the second equality utilizes the fact that $P_i(\bm{\phi})$ is self-adjoint with respect to $g$.
We observe that $\sigma(x,\xi)$ is a summation of type-$(1,1)$ tensors, each of which is self-adjoint and positive semi-definite with respect to $g$. Therefore, the symbol is invertible if and only if it is positive definite with respect to $g$, which is true if and only if 
\begin{align*}
    \mathrm{span}_{\mathbb{R}}\{P_1(\bm{\phi})_x\xi^\sharp,\dots,P_m(\bm{\phi})_x\xi^\sharp\} = T_x\M,
\end{align*}
for all $\xi\in T^*_x\M\backslash\{0\}$ and every $x\in\M$. Recall that the endomorphisms $P_i(\bm{\phi})$ are given by
\begin{align*}
    P_i(\bm{\phi}) &= h_i(\bm{\phi})^{-1}\circ g = \left((\mathrm{End}_{i}(\bm{\phi}))_*\tilde{g}_{\phi^i}\right)^{-1} \circ g, & \tilde{g}_{\phi^i}= \left(\int_0^1 \big((\Phi^i_\tau)^*g\big)^{-1}\,d\tau\right)^{-1},
\end{align*}
with $\Phi^i_t = \exp^g(t \grad \phi^i)$.
Since $\phi^i = \varepsilon\psi^i$, the family
$$
\varepsilon\mapsto\Phi^{i}_\tau = \exp^g(\varepsilon\,\tau\grad\psi^i),
$$
for all $\tau\in[0,1]$, is a smooth family of diffeomorphisms for all $\varepsilon\in[0,\varepsilon_0)$,  with $\Phi^i_\tau|_{\varepsilon=0}=\mathrm{id}$ and
$\partial_\varepsilon\big|_{\varepsilon=0}\Phi^i_\tau=\tau\grad\psi^i$.
Hence, for all $\tau\in[0,1]$,
\begin{align*}
    &(\Phi^i_\tau)^*g= g + \varepsilon \,\tau \,\mathfrak{L}_{\grad \psi^i} g + O(\varepsilon^2) 
    \implies((\Phi^i_\tau)^*g)^{-1}= g^{-1} - \varepsilon\,\tau\,g^{-1}(\mathfrak{L}_{\grad \psi^i} g)\, g^{-1} + O(\varepsilon^2)\\  &\implies \tilde{g}_{\phi^i}^{-1} = g^{-1} - \varepsilon\,g^{-1}(\mathfrak{L}_{\grad \psi^i} g)\, g^{-1}{\textstyle\int_0^1}\tau\,d\tau + O(\varepsilon^2)\implies \tilde{g}_{\phi^i}= g + \frac{1}{2}\varepsilon \,\mathfrak{L}_{\grad \psi^i} g + O(\varepsilon^2).
\end{align*}
where, since $\M$ is compact and $\bm{\psi}$ is fixed, the remainders are $O(\varepsilon^2)$ in the sense of Section \ref{sec:prelim}, and where, by \eqref{eq:killing}, $\mathfrak{L}_{\grad\psi^i}\,g=2\Hess\psi^i$.
Going back to the endomorphisms $P_i(\bm{\phi})$, we compute that
\begin{align*}
    &(\mathrm{End}_{i}(\bm{\phi}))_*\tilde{g}_{\phi^i} = g +  \varepsilon \Hess \psi^i - 2\varepsilon \sum_{j=1}^i \Hess\psi^j  + O(\varepsilon^2) \\
    &\implies P_i(\bm{\phi}) = I + \varepsilon g^{-1} \Hess \psi^i + 2\varepsilon \sum_{j=1}^{i-1}  g^{-1} \Hess\psi^j + O(\varepsilon^2)
\end{align*}
Since $\psi^1=0$, we have that $P_1(\bm{\phi}) = I$. Now observe that, 
\begin{align*}
    \mathrm{span}_{\mathbb{R}}\{P_1(\bm{\phi})_x\xi^\sharp,\dots,P_m(\bm{\phi})_x\xi^\sharp\} = \mathrm{span}_{\mathbb{R}}\{{Q}_1(\bm{\phi})_x\xi^\sharp,\dots,{Q}_m(\bm{\phi})_x\xi^\sharp\},
\end{align*}
where, for any $\varepsilon> 0$,  ${Q}_{i}(\bm{\phi})$ are defined by ${Q}_1(\bm{\phi})=P_1(\bm{\phi})$, and
$$
{Q}_{i}(\bm{\phi})=\varepsilon^{-1}\left(P_i(\bm{\phi})+(-1)^{i-1}P_1(\bm{\phi}) +2\sum_{j=2}^{i-1}(-1)^{i-j} P_j(\bm{\phi}) \right),
$$ 
for all $i\in\{2,\dots,m\}$. Direct computation shows that 
$
{Q}_i(\bm{\phi})= g^{-1}\Hess \psi^{i}+ O(\varepsilon).
$
Hence, if
\begin{align}\label{eq:zero_order_span}
    \mathrm{span}_{\mathbb{R}}\{\xi^\sharp,(g^{-1}\Hess\psi^2)_x\xi^\sharp,\dots,(g^{-1}\Hess\psi^{m})_x\xi^\sharp\}=T_x\M,
\end{align}
holds for all $x\in\M$ and all $\xi\in T_x^*\M\backslash\{0\}$, ellipticity for all sufficiently small $\varepsilon>0$ will follow. We then have the following assertion.
\begin{claim*}
    The spanning condition \eqref{eq:zero_order_span} is satisfied by the family $\{\psi^2,\dots,\psi^m\}$.
\end{claim*}
\begin{proof}
    First note that, for all $i\in\{1,\dots,\tbinom{N+1}{2}\}$, the Leibniz rule applied to the Hessian gives
\begin{align*}
    g^{-1}\Hess \psi^{2N+i+1} = g^{-1}\Hess (\beta^k\beta^l) &= \beta^l\, g^{-1}\Hess \beta^k + \beta^k\, g^{-1}\Hess \beta^l + g^{-1}(d\beta^l\otimes d\beta^k + d\beta^k\otimes d\beta^l),
\end{align*}
where $(k,l)=\iota(i)$. The first two terms always belong to the pointwise linear span of $g^{-1}\Hess\psi^{i+1}$ for $i\in\{1,\dots,N\}$, and so it suffices to satisfy the span condition
\begin{align*}
    \mathbb{R}\,\xi^\sharp +  \mathrm{span}_{\mathbb{R}}\{g^{-1}(d\beta^l\otimes d\beta^k + d\beta^k\otimes d\beta^l)_x\xi^\sharp~|~1\leq k\leq l \leq N\} = T_x\M,
\end{align*}
for all $x\in\M$ and all $\xi\in T_x^*\M\backslash\{0\}.$ Since $(\beta^1,\dots,\beta^N)$ is an immersion, the differentials $\{d\beta^1,\dots,d\beta^N\}$ span the cotangent bundle pointwise everywhere and, therefore, the set of symmetric $(0,2)$-tensor fields 
$$
\{d\beta^l\otimes d\beta^k + d\beta^k\otimes d\beta^l ~|~1\leq k\leq l \leq N\}
$$
spans the set of all symmetric $(0,2)$-tensor fields on $\M$ pointwise, i.e., 
\begin{align*}
    \mathrm{span}_{\mathbb{R}}\{(d\beta^l\otimes d\beta^k + d\beta^k\otimes d\beta^l)_x~|~1\leq k\leq l \leq N\} = \mathrm{Sym}^2(T_x^*\M),
\end{align*}
for all $x\in \M$. For any $v\in T_x\M$, define the symmetric $(0,2)$-tensor $S\in\mathrm{Sym}^2(T_x^*\M)$ by
\begin{align*}
    S:= \xi\otimes v^\flat +v^\flat\otimes\xi,
\end{align*}
which, as discussed above, belongs to the pointwise span of the family. We proceed to compute that
\begin{align*}
    (g_x^{-1}S)\xi^\sharp = g_x(v,\xi^\sharp)\xi^\sharp+|\xi|_g^2v \implies v =  \frac{1}{|\xi|_g^2}((g_x^{-1}S)\xi^\sharp - g_x(v,\xi^\sharp)\xi^\sharp).
\end{align*}
which, since $\xi\neq 0$, is well-defined. Therefore, 
\begin{align*}
    v\in \mathbb{R}\,\xi^\sharp +  \mathrm{span}_{\mathbb{R}}\{g^{-1}(d\beta^l\otimes d\beta^k + d\beta^k\otimes d\beta^l)_x\xi^\sharp~|~1\leq k\leq l \leq N\},
\end{align*}
and, since $v\in T_x\M$ was arbitrary, the statement of the claim follows.
\end{proof} 
Because the spanning condition \eqref{eq:zero_order_span} is an open property, an $O(\varepsilon)$ perturbation does not destroy it. Therefore, by homogeneity in $\xi$, there exists $\varepsilon_1\in(0,\varepsilon_0)$ such that, for all $\varepsilon\in(0,\varepsilon_{1})$, we have that 
\begin{align*}
    \mathrm{span}_{\mathbb{R}}\{P_1(\bm{\phi})_x\xi^\sharp,\dots,P_m(\bm{\phi})_x\xi^\sharp\} = \mathrm{span}_{\mathbb{R}}\{{Q}_1(\bm{\phi})_x\xi^\sharp,\dots,{Q}_m(\bm{\phi})_x\xi^\sharp\} = T_x\M,
\end{align*}
for all $\xi\in T^*_x\M\backslash\{0\}$ and every $x\in\M$, which implies that the principal symbol $\sigma(x,\xi)$ is invertible. Hence, for all $\varepsilon\in (0,\varepsilon_{1})$, the operator $\mathcal{L}(\bm{\phi})\circ \mathcal{L}^*(\bm{\phi})$ is elliptic.
Because the manifold is compact and without boundary, it follows that $\mathcal{L}(\bm{\phi})\circ \mathcal{L}^*(\bm{\phi})$ is Fredholm, i.e., has a finite dimensional kernel and co-kernel. In addition, because $\mathcal{L}(\bm{\phi})\circ \mathcal{L}^*(\bm{\phi})$ is self-adjoint, its kernel and co-kernel have the same dimension, so, to show invertibility, it suffices to show that its kernel is trivial. Now observe that, since $\mathcal{L}(\bm{\phi})\circ \mathcal{L}^*(\bm{\phi})$ is non-negative and self-adjoint, $\ker \mathcal{L}(\bm{\phi})\circ \mathcal{L}^*(\bm{\phi})=\ker\mathcal{L}^*(\bm{\phi})$. Moreover, we have that
\begin{align*}
    \ker\mathcal{L}^*(\bm{\phi})=\{X\in \X~|~\Div(P_i(\bm{\phi}) X) = 0,\,\forall i\in\{1,\dots,m\}\}.
\end{align*}
Utilizing linearity of the divergence operator, we re-write the conditions once more in terms of the endomorphisms ${Q}_i$, obtaining that, as long as $\varepsilon \in(0,\varepsilon_{1})$,
\begin{align*}
    \Div(P_i(\bm{\phi}) X) = 0, \,\forall i\in\{1,\dots,m\} \iff \Div({Q}_i(\bm{\phi}) X)=0, \,\forall i\in\{1,\dots,m\}.
\end{align*}
Therefore, if we define the operator $\mathcal{T}^*(\bm{\phi})$ by
\begin{align*}
    \mathcal{T}^*(\bm{\phi})[X]= -\bigl(\Div({Q}_1(\bm{\phi}) X),\dots,\Div( {Q}_m(\bm{\phi}) X)\bigr),
\end{align*}
then, for all $\varepsilon\in(0,\varepsilon_{1})$, we have that
\begin{align*}
    \mathrm{ker} \mathcal{L}^*(\bm{\phi}) = \mathrm{ker}\mathcal{T}^*(\bm{\phi}).
\end{align*}
Now observe that the operator $\mathcal{T}^*(\bm{\phi})$ can be written as
\begin{align*}
    \mathcal{T}^*(\bm{\phi}) &= \mathcal{T}_0^* + \mathcal{T}_{1}^*(\bm{\phi}), 
\end{align*}
where $\mathcal{T}_0^*$ and $\mathcal{T}_{1}^*(\bm{\phi})$ are given by
\begin{align*}
    \mathcal{T}_0^*[X] &= -\bigl(\Div(X),\Div(g^{-1}\Hess \psi^2\,X), \dots, \Div(g^{-1}\Hess \psi^m\,X)\bigr),\\
    \mathcal{T}_{1}^*(\bm{\phi})[X]&=  -\bigl(\Div(({Q}_1(\bm{\phi})-I)X),\Div(({Q}_2(\bm{\phi})-g^{-1}\Hess \psi^2)\,X), \dots, \Div(({Q}_m(\bm{\phi})-g^{-1}\Hess \psi^m)\,X)\bigr),
\end{align*}
and $\mathcal{T}_{1}^*(\bm{\phi}):H^{s+1}(T\M)\rightarrow H_m^{s}(\M)$ is a bounded linear operator for all $\varepsilon\in(0,\varepsilon_{1})$ and all $s\geq 0$. Hence,
\begin{align*}
    \mathcal{T}(\bm{\phi})\circ\mathcal{T}^*(\bm{\phi}) &= \mathcal{T}_0 \circ \mathcal{T}_0^* + \mathcal{S}(\bm{\phi}),
\end{align*}
where the operator $\mathcal{S}(\bm{\phi})$ is defined by
\begin{align*}
    \mathcal{S}(\bm{\phi})&:=\mathcal{T}_0\circ \mathcal{T}_{1}^*(\bm{\phi}) + \mathcal{T}_{1}(\bm{\phi}) \circ \mathcal{T}_{0}^* + \mathcal{T}_{1}(\bm{\phi})\circ \mathcal{T}_{1}^*(\bm{\phi}).
\end{align*}
The same ellipticity analysis we did for $\mathcal{L}(\bm{\phi})\circ\mathcal{L}^*(\bm{\phi})$ shows that $\mathcal{T}_0 \circ \mathcal{T}_0^*$ is Fredholm. Moreover, as can be seen above, the coefficients of $\mathcal{T}_{1}(\bm{\phi})$ are, uniformly, of $O(\varepsilon)$. Therefore, the operator norm of $\mathcal{S}(\bm{\phi}):H^2(T\M)\rightarrow H^0(T\M)$ can be made arbitrarily small by choosing $\varepsilon>0$ sufficiently small. Hence, we can invoke \cite[Theorem 5.22, Chapter IV]{kato1966perturbation} to conclude that there exists $\varepsilon_{2}\in(0,\varepsilon_1)$ such that, for all $\varepsilon\in(0,\varepsilon_2)$, we have that
\begin{align*}
    \dim \ker \mathcal{T}(\bm{\phi})\circ\mathcal{T}^*(\bm{\phi}) \leq  \dim \ker \mathcal{T}_0\circ \mathcal{T}_0^*.
\end{align*}
We then have the following claim.
\begin{claim*}
$\ker \mathcal{T}_0\circ \mathcal{T}_0^*= \{0\}$. 
\end{claim*}
\begin{proof}
    Because $\mathcal{T}_0 \circ \mathcal{T}_0^*$ is self-adjoint and non-negative, $\ker \mathcal{T}_0 \circ \mathcal{T}_0^*=\ker \mathcal{T}_0^*$.
We now prove that $\ker \mathcal{T}_0^*= \{0\}$.
To this end, if $X\in\ker \mathcal{T}_0^*$, then $X$ must simultaneously satisfy
\begin{align*}
    \Div(X) &= 0, & \Div(g^{-1}\Hess \psi^i\,X)&=0
\end{align*}
for all $i\in\{2,\dots,m\}$.
For any $\chi\in\Cinf$, consider the pairing
\begin{align*}
    \int_{\M} g([\grad\psi^i,\grad\chi],X)\,d\mu_g,
\end{align*}
and compute that
\begin{align*}
    [\grad\psi^i,\grad\chi] = g^{-1}\Hess \chi\,\grad\psi^i - g^{-1}\Hess \psi^i\,\grad \chi.
\end{align*}
In addition, for any vector field $X\in\X$, we have that
\begin{align*}
    X (g(\grad \psi^i,\grad \chi)) &= g(\nabla_X\grad \psi^i,\grad \chi) + g(\grad \psi^i,\nabla_X\grad \chi)\\
    &= \Hess \chi (\grad \psi^i,X) + \Hess\psi^i (\grad \chi,X),
\end{align*}
which implies that
\begin{align*}
    \grad(g(\grad \psi^i,\grad \chi)) = g^{-1}\Hess \chi\,\grad\psi^i + g^{-1}\Hess \psi^i\,\grad \chi.
\end{align*}
Therefore, we obtain that
\begin{align*}
    [\grad\psi^i,\grad\chi] = \grad (g(\grad \psi^i,\grad \chi)) -2 g^{-1}\Hess  \psi^i \grad\chi.
\end{align*}
and we have that
\begin{align*}
    \int_{\M} g([\grad\psi^i,\grad\chi],X)\,d\mu_g& = \int_{\M} g(\grad (g(\grad \psi^i,\grad \chi)),X)\,d\mu_g- 2 \int_{\M} g(g^{-1}\Hess  \psi^i \grad\chi,X)\,d\mu_g.
\end{align*}
By applying the divergence theorem, we compute that
\begin{align*}
   &\int_{\M} g(g^{-1}\Hess  \psi^i \grad\chi,X)\,d\mu_g = \int_{\M} g(\grad\chi,g^{-1}\Hess  \psi^i X)\,d\mu_g = -  \int_{\M} \chi\,\Div(g^{-1}\Hess  \psi^i X) \,d\mu_g = 0, \\
    &\int_{\M} g(\grad (g(\grad \psi^i,\grad \chi)),X)\,d\mu_g= -\int_{\M} g(\grad \psi^i,\grad \chi)\Div(X)\,d\mu_g = 0
\end{align*}
for all $\chi\in\Cinf$, and all $i\in\{2,\dots,m\}$, where we used the assumption that $X\in\ker \mathcal{T}_0^*$. Therefore, if $X\in\ker \mathcal{T}_0^*$, then $X$ is $L^2(\mu_g)$ orthogonal to 
\begin{align*}
    \mathfrak{V}:=\mathrm{span}_{\mathbb{R}}\{[\grad\psi^i,\grad\chi]~|~\chi\in\Cinf, \, i\in\{2,\dots,m\}\}.
\end{align*}
However, we have already established in \cite[Proposition 1]{abdelgalil2026holonomy} that $\mathfrak{V}=\X$ whenever the family $\{\psi^i\}$ contains the monomials 
$
\{\beta^i,(\beta^i)^2,(\beta^i)^3\}
$, and therefore $X=0$.
\end{proof} 
The above claim implies that, 
for any $\varepsilon\in(0,\varepsilon_2)$,  
$$
\{0\}= \ker \mathcal{T}_0^* = \ker \mathcal{T}_0 \circ \mathcal{T}_0^* = \ker \mathcal{T}(\bm{\phi})\circ\mathcal{T}^*(\bm{\phi})=\ker \mathcal{T}^*(\bm{\phi})=\ker \mathcal{L}^*(\bm{\phi}).
$$
Fixing any $\varepsilon\in(0,\varepsilon_{2})$, it follows that the choice $\bm{\phi}_\star = \varepsilon\,\bm{\psi}$ has the claimed properties, i.e., that $\bm{\phi}_\star,\bm{\phi}_\star^c\in O_m$, and that  
$
\mathcal{L}(\bm{\phi}_\star)\circ \mathcal{L}^*(\bm{\phi}_\star):H^{s+2}(T\M)\rightarrow H^{s}(T\M)
$
is an elliptic and invertible linear operator for every $s\geq 0$.
\end{proof}
We now provide the closing argument. Let $m$ and $\bm{\phi}_{\star}$ be given as in Proposition \ref{prop:ellipticity} and let $\mathcal{U}\subset O_m$ be given as in Proposition \ref{prop:tame_inverse}. For all $\bm{\phi}\in \mathcal{U}$, let
\begin{align*}
    d\mathrm{End}_{m}^\dagger(\bm{\phi}):= \mathcal{C}^{-1}(\bm{\phi})\circ \mathcal{L}^\dagger(\bm{\phi})\circ \mathcal{R}^{-1}(\bm{\phi}).
\end{align*}
Then, with $\mathrm{End}_{m}|_{\mathcal{U}}$ being the restriction of $\mathrm{End}_{m}$ to $\mathcal{U}$, the family of operators
\begin{align*}
    d\mathrm{End}_{m}^\dagger: (\mathrm{End}_{m}|_{\mathcal{U}})^*T\Diff \rightarrow \CinfV{m}
\end{align*}
is smooth tame, and, for all $\bm{\phi}\in \mathcal{U}$, we have that 
\[
d\mathrm{End}_{m}(\bm{\phi})\circ d\mathrm{End}_{m}^\dagger(\bm{\phi})=\mathrm{Id}_{T_{\mathrm{End}_{m}(\bm{\phi})}\Diff}.
\]
Hence, from \cite[Part~III, Theorem 1.1.3]{Hamilton1982}, the image under $\mathrm{End}_{m}$ of every open neighborhood of $\bm{\phi}_\star$ in $\mathcal{U}$ contains an open neighborhood of $\mathrm{End}_{m}(\bm{\phi}_\star)$ in $\Diff$. The conclusion of the theorem follows from the discussion at the beginning of Section \ref{sec:technical}, with $r=2m=2+4N+2\tbinom{N+1}{2}$. \hfill $\Box$

\section{Concluding Remark}
An alternative proposal for Balantine-type factorizations into factors generated by constant gradient vector fields \cite{abdelgalil2026holonomy}, when contrasted with the choice of OMT maps, reveals a significant challenge. When composing gradient flows of the potentials $\phi^i$,
\begin{align*}
    \mathrm{End}_{\ell}^{\grad}:\bm{\phi}\mapsto e^{\grad\phi^1}\circ\cdots \circ e^{\grad\phi^\ell},
\end{align*}
a significant departure from composing OMT maps can already be seen when considering the differential of only one entry,
\begin{align*}
    d\mathrm{End}_{1}^{\grad}(\phi)[\delta] = 
    \mathcal B_\phi(\delta) \circ e^{\grad\phi}, &  \mbox{ where }    \mathcal B_\phi(\delta) =\int_0^1 (e^{\tau \grad\phi})_*\grad \delta\, d\tau.
\end{align*}
Unlike the case of OMT maps, this does not reduce to a local differential operator, and as a result the elliptic argument used earlier to show the existence of a tame right inverse does not apply.

\bibliography{refs}
\bibliographystyle{plain}

\end{document}